\documentclass[12pt]{amsart}

\usepackage{amsmath, amsthm, amssymb, mathrsfs, latexsym, bbm, xypic, epsf, setspace}

\newif\ifpdf
\ifx\pdfoutput\undefined
\pdffalse % we are not running PDFLaTeX
\else
\pdfoutput=1 % we are running PDFLaTeX
\pdftrue
\fi

\ifpdf
\usepackage[pdftex]{graphicx}
\else
\usepackage{graphicx}
\fi

\newtheorem{thm}{Theorem}[section]

\newtheorem{lem}[thm]{Lemma}
\newtheorem{rem}[thm]{Remark}

\newtheorem{ass}[thm]{Assumption}
\def\Ind#1#2{#1\setbox0=\hbox{$#1x$}\kern\wd0\hbox to 0pt{\hss$#1\mid$\hss}
\lower.9\ht0\hbox to 0pt{\hss$#1\smile$\hss}\kern\wd0}
\def\Notind#1#2{#1\setbox0=\hbox{$#1x$}\kern\wd0\hbox to 0pt{\mathchardef
\nn="3236\hss$#1\nn$\kern1.4\wd0\hss}\hbox to 0pt{\hss$#1\mid$\hss}\lower.9\ht0
\hbox to 0pt{\hss$#1\smile$\hss}\kern\wd0}

\def\cl{{\rm cl}}

\def\K{\mathcal{K}}
\def\cN{\mathcal{N}}
\def\M{\mathcal{M}}
\def\A{\mathcal{A}}
\def\C{\mathcal{C}}

\def\Cb{\bar{\mathcal{C}}}
\def\S{\mathcal{S}}
\def\hC{\hat{\mathcal{C}}}
\def\G{\mathcal{G}}

\begin{document}

\title[The geometry of Hrushovski constructions, II.]{The geometry of Hrushovski constructions, II. \\ The strongly minimal case.\footnote{Version: 18 March 2011}}

\author{David M. Evans}
\address{School of Mathematics, UEA, Norwich NR4 7TJ, UK.}
\email{d.evans@uea.ac.uk}
\author{Marco S. Ferreira}
\address{School of Mathematics, UEA, Norwich NR4 7TJ, UK.}
\email{masferr@gmail.com}

\begin{abstract}
We investigate the isomorphism types of combinatorial geometries arising from Hrushovski's flat strongly minimal structures and answer some questions from Hrushovski's original paper.

\medskip

\end{abstract}

\maketitle

\section{Introduction}

In this paper, we investigate the isomorphism types of combinatorial geometries arising from Hrushovski's flat strongly minimal structures and answer some questions from Hrushovski's original paper \cite{EH}. It is a sequel to \cite{MFDE1}, but can be read independently of it. In order to describe the main results it will be convenient to summarise some of the results from the previous paper.

Suppose $L$ is a relational language with, for convenience, all relation symbols of arity at least 3 and at most one relation symbol of each arity. Denote by $k(L)$ the maximum of the arities of the relation symbols in $L$ (allowing $k(L)$ to be $\infty$ if this is unbounded). The basic Hrushovski construction defines the \textit{predimension} of a finite $L$-structure to be its size minus the number of atomic relations on the structure. The class $\C_0(L)$ consists of the finite $L$-structures in which this is non-negative on all substructures. There is then an associated notion of \textit{dimension} $d$ and the notion of \textit{self-sufficiency} (denoted by $\leq$) of a substructure. All of this is reviewed in detail in Section 2 below. The class $(\C_0, \leq)$ has an associated \textit{generic structure} $\M_0(L)$ which also carries a dimension function $d$ giving it the structure of  an infinite-dimensional pregeometry. The associated (combinatorial) geometry is denoted by $G(\M_0(L))$. 

In \cite{MFDE1} we showed that:

\begin{enumerate}
\item[(1)] The collection of finite subgeometries of $G(\M_0(L))$ does not depend on $L$ (Theorem 3.8 of \cite{MFDE1}).
\item[(2)] For languages $L, L'$, the geometries $G(\M_0(L))$ and $G(\M_0(L'))$ are isomorphic iff the maximum arities $k(L)$ and $k(L')$ are equal. (Theorem 3.1 of \cite{MFDE1} for $\Leftarrow$ and see also Section 4.2 here; Theorem 4.3 of \cite{MFDE1} gives $\Rightarrow$.)
\item[(3)] The localization of $G(\M_0(L))$ over any finite set is isomorphic to $G(\M_0(L))$ (Theorem 5.5 of \cite{MFDE1}).
\end{enumerate}

For the strongly minimal set construction of \cite{EH}, one takes a certain function $\mu$ (see section 2 here) and considers a subclass $\C_\mu(L)$ of $\C_0(L)$. For appropriate $\mu$ there is a generic structure $\M_\mu(L)$ for the class $(\C_\mu(L), \leq)$ which is strongly minimal. The dimension function given  by the predimension is the same as the dimension in the strongly minimal set and we are interested in the geometry of this. Our main result here is that this process of `collapse' is irrelevant to the geometry: under rather general conditions on $\mu$ we prove:

\begin{enumerate}
\item[(4)] The geometry $G(\M_\mu(L))$ of the strongly minimal set is isomorphic to the geometry $G(\M_0(L))$ (Theorem \ref{main}).
\end{enumerate}

Sections 5.1 and 5.2 of Hrushovski's paper \cite{EH} give variations on the construction which produce strongly minimal sets with geometries different from the $G(\M_0(L))$. However, we show, answering a question from \cite{EH} (see also Section 3 of \cite{Hasson}):

\begin{enumerate}
\item[(5)] the geometries of the strongly minimal sets in Sections 5.1 and 5.2 of \cite{EH} have localizations (over a finite set) which are isomorphic to one of the geometries $G(\M_0(L))$ (for appropriate $L$) (see Section 4.1 here).
\end{enumerate}

The first version of the result in (4) was proved by the second Author in his thesis \cite{MFThesis}: this was for the case where $L$ has a single 3-ary relation symbol (as in the original paper \cite{EH}). The somewhat different method of proof used in Sections 3 and 4 here was found later. It has the advantage of being simpler and more readily adaptable to generalization and proving  the result in (5), however, the class of $\mu$-functions to which it is applicable is slightly more restricted than the result from \cite{MFThesis}: Theorem 6.2.1 of \cite{MFThesis} assumes only that $\mu \geq 1$.

\medskip

In summary, for each $k = 3, 4, \ldots, \infty$ we have a countably-infinite dimensional geometry $\G_k$ isomorphic to $G(\M_0(L))$ where $L$ has maximum arity $k$, and these are pairwise non-isomorphic. The geometry of each  of the new (countable, saturated) strongly minimal sets in \cite{EH} has a localization isomorphic to one of these $\G_k$. Thus, whilst there is some diversity amongst the strongly minimal structures which can be produced by these constructions, the range of geometries which can be produced appears to be rather limited. It would therefore be very interesting to have a characterization of the geometries $\G_k$ in terms of a `geometric' condition (such as  flatness, as in 4.2 of \cite{EH}, for example) and a condition on the automorphism group (such as homogeneity, but possibly with a stronger assumption).

\medskip

\noindent\textit{Acknowledgement:\/} Some of the results of this paper were produced whilst the second Author was supported as an Early Stage Researcher by the Marie Curie Research Training Network MODNET, funded by grant MRTN-CT-2004-512234 MODNET from the CEC. We thank the Referee for drawing our attention to the need to consider local isomorphisms in Theorem \ref{main}.

\section{Hrushovski constructions}\label{sec2}

We give a brief description of Hrushovski's constructions from \cite{EH}. Other presentations can be found in  \cite{FW} and  \cite{B&S}. The book \cite{AP2} of Pillay contains all necessary background material on pregeometries and model theory. The notation, terminology and level of generality is mostly consistent with that used in \cite{MFDE1}.

\subsection{Predimension and pregeometries}

Let $L$ be a relational language consisting of relation symbols  $(R_i : i \in I)$ with $R_i$ of arity $n_i \geq 3$ (and $\vert I \vert \geq 1$). We suppose there are only finitely many relations of each arity here. 

We work with $L$-structures $A$ where each each $R_i$ is symmetric: so we regard the interpretation $R_i^A$ of $R_i$ in $A$ as a set of $n_i$-sets. (By modifying the language, the arguments we give below can be adapted to deal with the case of $n_i$-tuples of not-necessarily-distinct elements: see Section 4.3 here.)

For finite $A$ we let the predimension of $A$ be $\delta(A) = \vert A \vert - \sum_{i \in I} \vert  {R_i}^A\vert$ (of course this depends on $L$ but this will be clear from the context). 

We let 
$\C_0(L)$ be the set of finite $L$-structures $A$ such that $\delta(A') \geq 0$ for all $A'  \subseteq A$. 

Suppose $A \subseteq B \in \C_0(L)$.  We write $A \leq B$ and say that $A$ is \textit{self-sufficient} in $B$  if for all $B'$ with $A \subseteq B' \subseteq B$ we have $\delta(A) \leq \delta(B')$. We will assume that the reader is familiar with the basic properties (such as transitivity) of this notion.

Let $\Cb_0(L)$  be the class of $L$-structures all of whose finite substructures lie in $\C_0(L)$. We can extend the notion of self-sufficiency to this class in a natural way.

Note that if $A \subseteq B \in \Cb_0(L)$ is finite then there is a finite $A'$ with $A \subseteq A' \subseteq B$ and $\delta(A')$ as small as possible. In this case $A' \leq B$ and it can be shown that there is a smallest finite set $C \leq B$ with $A \subseteq C$.  We define the \textit{dimension} $d_B(A)$ of $A$ (in $B$)  to be the minimum value of $\delta(A')$ for all finite subsets $A'$ of $B$ which contain $A$. 

We define the $d$-\textit{closure} of $A$ in $B$ to be: 
$$\cl_B(A)=\{c\in B:d_{B}(Ac)=d_{B}(A)\}$$ 
where, as usual,  $Ac$ is shorthand for $A\cup \{c\}$.

These notions can be relativized. If $A, C \subseteq B \in \C_0(L)$ define $\delta(A/C)$ to be $\delta(A\cup C)- \delta(C)$ and $d_B(A/C) = d_B(A\cup C) - d_B(C)$: the (pre)dimension \textit{over} $C$. In all of this notation, we will suppress the subscript for the ambient structure $B$ if it is clear from the context.

We can coherently extend the definition of $d$-closure to infinite subsets $A$ of $B$ by saying that the $d$-closure of $A$ is the union of the $d$-closures of finite subsets of $A$. 
It can be shown that  $(B,\cl_B)$ is a pregeometry and the dimension function (as cardinality of a basis) equals $d_{B}$ on finite subsets of $B$. We use the notation $PG(B)$ instead of $(B,\cl_B)$, and denote by $G(B)$ the associated (combinatorial) geometry: so the elements of $G(B)$ are the sets $\cl_B(x)\setminus\cl_B(\emptyset)$ for $x \in B \setminus \cl_B(\emptyset)$ and the closure on $G(B)$ is that induced by $\cl_B$. 
Note that if $A \leq B \in \Cb_0(L)$ then for $X \subseteq A$ we have $d_A(X) = d_B(X)$. Thus $G$ can be regarded as a functor from $(\Cb_0(L), \leq)$ to the category of geometries (with embeddings of geometries as morphisms).

If $Y \subseteq B \in \Cb_0(L)$ then the \textit{localization} of $PG(B)$ over $Y$ is the pregeometry with closure operation $\cl_B^Y(Z) = \cl_B(Y\cup Z)$. The corresponding geometry is denoted by $G_Y(B)$. Note that the dimension function here is given by the relative dimension $d_B(./Y)$.

It will be convenient to fix a first order language for the class of pregeometries.  A reasonable choice for this is the language $LPI=\{I_n:n\geq 1\}$ where each $I_n$ is an $n$-ary relational symbol. A pregeometry $(P,\cl)$ will be seen as a structure in this language by taking $I_n^P$ to be the set of independent $n$-tuples in $P$.  Notice that we can recover a pregeometry just by knowing its finite independent sets. Note also that the isomorphism type of a pregeometry is determined by the isomorphism type of its associated geometry and the size of the equivalence classes of interdependence. In the case where these are all countably infinite, it therefore makes no difference whether we consider the geometry or the pregeometry.

\subsection{Self-sufficient amalgamation classes}

If $B_1, B_2 \in \Cb_0(L)$ have a common substructure $A$ then the \textit{free amalgam} $E$ of $B_1$ and $B_2$ over $A$ consists of the disjoint union of $B_1$ and $B_2$ over $A$ and  $R_i^E = R_i^{B_1} \cup R_i^{B_2}$ for each $i \in I$. It is well known that if $A \leq B_1$ then $B_2 \leq E$, so $E \in \Cb_0(L)$, and $(\C_0, \leq)$ is an \textit{amalgamation class}.  It can also be shown that if $A$ is $d$-closed in $B_1$ then $B_2$ is $d$-closed in $E$. 

\smallskip

Suppose $Y \leq Z \in \C_0(L)$ and $Y \neq Z$.  Following \cite{EH}, we say that this is an \textit{algebraic extension} if $\delta(Y) = \delta(Z)$. It is a simply algebraic extension if also $\delta(Z') > \delta(Y)$ whenever $Y \subset Z' \subset Z$. It is a minimally simply algebraic (msa) extension  if additionally $Y' \subseteq Y'\cup (Z\setminus Y)$ is not simply algebraic whenever $Y' \subset Y$. 

It can be shown that for each $n \geq 0$ there are arbitrarily large msa extensions $Y \leq Z$ in $\C_0(L)$ with $\delta(Y) = n$ (this make use of the fact that at least one of the relations $R_i$ has arity at least 3, which is certainly covered by our assumptions on $L$).

The following is trivial, but crucial for us:

\begin{lem}\label{21}
Suppose $Y \leq Z$ is a msa extension.  Then for every $y \in Y$ there is some $w \in \bigcup_{i \in I} R_i^Z$ and $z \in Z\setminus Y$ such that $y, z \in w$. Moreover, if $Z\setminus Y$ is not a singleton and $z \in Z\setminus Y$, then there are at least two  elements of $\bigcup_{i \in I} R_i^Z$ which contain $z$.
\end{lem}

\begin{proof}
Suppose this does not hold for some $y \in Y$. Let $Y' = Y\setminus \{y\}$. Then for every $U \subseteq Z\setminus Y$ we have $\delta(Y'\cup U) - \delta(Y') = \delta(Y \cup U) - \delta(Y)$. So $Y' \subseteq Y'\cup (Z\setminus Y)$ is simply algebraic: contradiction. Similarly, for the `moreover' part, if $z$ is in at most one relation in $\bigcup_{i \in I} R_i^Z$, then $\delta(Z\setminus\{z\}) \leq \delta(Z)$, which contradicts the simple algebraicity.
\end{proof}

We let $\mu$ be a function from the set of isomorphism types of minimally simply algebraic extensions in $\C_0(L)$ to the non-negative integers. The subclass $\C_\mu(L)$ consists of structures in $\C_0(L)$ which, for each msa $Y\leq Z$ in $\C_0(L)$, omit the atomic type consisting of $\mu(Y,Z)+1$ disjoint copies of $Z$ over $Y$. 

We will work with $\mu$ where the following holds:

\begin{ass}[Assumed Amalgamation Lemma]\label{aalemma}
\begin{enumerate}
\item[(i)] If $A \leq B_1, B_2 \in \C_\mu(L)$ and the free amalgam of $B_1$ and $B_2$ over $A$ is not  in $\C_\mu(L)$, then there exists $Y \subseteq A$ and minimally simply algebraic extensions $Y \leq Z_i \in B_i$ (for $i = 1,2$) which are isomorphic over $Y$ and $Z_i\setminus Y \subseteq B_i\setminus A$.
\item[(ii)]  The class $(\C_\mu(L), \leq)$ is an amalgamation class (see below).
\end{enumerate}
\end{ass}

Note that (ii) here follows from (i) (cf. the proof of Lemma 4 in \cite{EH}), and by Section 2 of \cite{EH}, (i) holds if $\mu(Y,Z) \geq \delta(Y)$ for all msa $Y\leq Z$ in $\C_0(L)$. 

\subsection{Generic structures and their geometries}

Suppose $\A$ is a subclass of $\C_0(L)$ such that $(\A, \leq)$ is an amalgamation class: meaning that if $B \in \A$  and $A \leq B$ then $A \in \A$, and if $A \leq B_1, B_2 \in \A$ then there is $C \in \A$ and embeddings $f_i : B_i \to C$ with $f_i(B_i) \leq C$ and $f_1 \vert A = f_2 \vert A$. Then there is a countable structure $\M \in \Cb_0(L)$ satisfying the following conditions:\begin{enumerate}
\item [(G1)]  $\mathcal{M}$ is the union of a chain $A_0 \leq A_1 \leq A_2 \leq \cdots$ of structures in $\A$.
\item [(G2)] (extension property) If $A\leq\mathcal M$ and $A\leq B\in\mathcal A$ then there exists an embedding $g:B\to\mathcal M$ such that  $g(B)\leq\mathcal M$ and $g(a) = a$ for all $a \in A$.
\end{enumerate}
We refer to $\M$ as the \textit{generic structure} of the amalgamation class $(\A, \leq)$: it is determined up to isomorphism by the properties  G1 and G2  (and G1 is automatic for countable structures in $\Cb_0(L)$). Of course, Hrushovski's strongly minimal sets are the generic structures $\M_\mu(L)$ for the amalgamation classes $(\C_\mu(L), \leq)$. We will compare the geometries of these with that of the generic structure $\M_0(L)$ for the amalgamation class $(\C_0(L), \leq)$.

Suppose $(\A, \leq)$ and $(\A', \leq)$ are amalgamation classes, as above. We refer to the following as  the Isomorphism Extension Property, and denote it by $\A \rightsquigarrow \A'$.
\begin{enumerate}
\item[(*)] Suppose $A \in \A$, $A' \in \A'$ and $f : G(A) \to G(A')$ is an isomorphism of geometries, and $A \leq B \in \A$. Then there is $B' \in \A'$ with $A' \leq B'$ and an isomorphism $f' : G(B) \to G(B')$ which extends $f$.
\end{enumerate}

\begin{lem}\label{isolemma}
Suppose $(\A, \leq)$ and $(\A', \leq)$ are amalgamation classes with generic structures $\M$, $\M'$ respectively. Suppose that both extension properties $\A \rightsquigarrow \A'$ and $\A' \rightsquigarrow \A$ hold. Then the geometries $G(\M)$ and $G(\M')$ are isomorphic.
\end{lem}

\begin{proof}
We have already remarked that if $A \leq \M$, then  the dimension of a subset of $A$ is the same whether computed in $A$ or in $\M$. Thus $G(A)$ is naturally a substructure of $G(\M)$. We claim that the set $\S$ of geometry-isomorphisms 
\[ f : G(A) \to G(A') \]
where $A \leq \M, A' \leq \M'$ are finite
is a back-and-forth system between $G(\M)$ and $G(\M')$. Indeed (for the `forth'), given such an $f : G(A) \to G(A')$ and $A \leq B \leq \M$, there is $A' \leq B' \in \A'$ and an isomorphism $f' : G(B) \to G(B')$ extending $f$, by our assumption $\A \rightsquigarrow \A'$. The extension property G2 in $\M'$ means that we can take $B' \leq \M'$, as required. Similarly we obtain the `back' part from $\A' \rightsquigarrow \A$ and G2 in $\M$.
It follows that  $G(\M)$ and $G(\M')$ are isomorphic.
\end{proof}

\begin{rem} \label{isorem}\rm 
We can adapt this slightly to give a criterion for local isomorphism of $G(\M)$ and $G(\M')$. This is really just about adding parameters to the language, but we make it explicit. Suppose $X \in \A$ and $X' \in \A'$ are fixed finite structures with $X \leq \M$ and $X' \leq \M'$. We write $\A(X) \rightsquigarrow \A'(X')$ for the statement:
\begin{enumerate}
\item[]
Suppose $X \leq A \in \A$, $X' \leq A' \in \A'$ and $f : G_X(A) \to G_X'(A')$ is an isomorphism of geometries, and $A \leq B \in \A$. Then there is $B' \in \A'$ with $A' \leq B'$ and an isomorphism $f' : G_X(B) \to G_{X'}(B')$ which extends $f$.
\end{enumerate}
Here, recall that $G_X(A)$ is the localization of $G(A)$ over $X$, as defined in Section 2.1. It follows as in  Lemma \ref{isolemma} that if $\A(X) \rightsquigarrow \A'(X')$ and $\A'(X') \rightsquigarrow \A(X)$ hold, then $G_X(\M)$ and $G_{X'}(\M')$ are isomorphic.
\end{rem}

\section{Isomorphism of the strongly minimal set geometries}

Throughout, $(\C_0(L), \leq)$ and $(\C_\mu(L), \leq)$ are the  amalgamation classes from the previous section. Note that $(\C_0(L), \leq)$ is an amalgamation class and we are \textit{assuming} that the amalgamation lemma \ref{aalemma} holds for $\C_\mu(L)$. We denote the generic structures by $\M_0(L)$ and $\M_\mu(L)$ respectively: so the latter is Hrushovski's strongly minimal set $D(L,\mu)$.  The geometries are denoted by $G(\M_0(L))$  and $G(\M_\mu(L))$.  We have already remarked that there are arbitrarily large msa extensions  $\emptyset \leq Z$ in $\C_0(L)$, so if $\mu(\emptyset, Z) > 0$ for infinitely many of these, then $\cl_{\M_\mu(L)}(\emptyset)$ is infinite as it contains a copy of each such $Z$. Similarly,  there are infinitely many msa extensions $Y \leq Z \in \C_0(L)$ with $\delta(Y) = 1$, so if  $\mu(Y, Z) > 0$ for all of these, and $X \leq \M_\mu(L)$ is a singleton, then $\cl_{\M_\mu(L)}(X)$ is infinite. Our main result is:

\begin{thm}\label{main} Suppose \ref{aalemma} holds and $\mu(Y,Z) \geq 2$ for all msa $Y \leq Z \in \C_0(L)$ with $\delta(Y) \geq 2$ and $\mu(Y,Z) \geq 1$ when $\delta(Y) = 1$. Suppose that $X \leq \M_\mu(L)$ is  finite and   $\cl_{\M_\mu(L)}(X)$ is infinite. Then $G_X(\M_\mu(L))$ and  $G(\M_0(L))$  are isomorphic geometries. In particular, $G(\M_\mu(L))$ and  $G(\M_0(L))$ are locally isomorphic.
\end{thm}

\begin{proof}
Note that $X$ can be taken as $\emptyset$ or a singleton, by the remarks preceding the theorem. 
We need to verify that the isomorphism extension property of Lemma \ref{isolemma} and Remarks \ref{isorem} holds in both directions. The main part will be to show that 
 $\C_0(L) \rightsquigarrow \C_\mu(L)(X)$. In fact, because of the symmetry of the argument, it will be convenient to take an arbitrary finite $W \leq \M_0(L)$ and show that $\C_0(L)(W) \rightsquigarrow \C_\mu(L)(X)$. 
 
 \medskip

So suppose we are given $W \leq A \leq B \in \C_0(L)$ and $X \leq A' \in \C_\mu(L)$ with an isomorphism $f : G_W(A) \to G_X(A')$. We want to find $B' \in \C_\mu(L)$ with $A' \leq B'$ and an isomorphism $f' : G_W(B) \to G_X(B')$ extending $f$. The main point will be to ensure that each point of $B'\setminus (A' \cup \cl_{B'}(X))$ is involved in only a small number of relations, and this gives us control over the msa extensions in $B'$.

Let $A_0 = \cl_A(W)$ and let $A_1,\ldots, A_r$ be the $d$-dependence classes (over $W$)  on $A\setminus A_0$: the latter are the points of $G_W(A)$. Similarly let $B_0 = \cl_B(W)$ and $B_1, \ldots, B_s$ the $d$-dependence classes (over $W$) on $B\setminus B_0$, with $A_i \subseteq B_i$ for $i = 1,\ldots, r$. List the relations on $B$ which are not contained in $A$ or some $B_0\cup B_j$ as $\rho_1,\ldots, \rho_t$. So these are finite sets. Let $A_0' = \cl_{A'}(X)$ and $A_1',\ldots, A_r'$ be the classes of $d$-dependence (over $X$)  on $A'\setminus A_0'$, labelled so that $f(A_i) = A_i'$. We construct $B_0', B_1', \ldots, B_s'$ with $A_i' \subseteq B_i'$ for $i = 0, \ldots, r$, and $B' = \bigcup_{i = 0}^s B_i'$ in the steps below. 

\medskip

\noindent\textit{Terminology:\/} If $u, v \in E \in \C_0(L)$, say that $u,v$ are \textit{adjacent in $E$} if there exists $w \in \bigcup_i R_i^E$ such that $u,v \in w$.

\medskip

\noindent\textit{Step 1:\/} Construction of $A'' = A' \cup B_0' \in \C_\mu(L)$. \newline
Take $A_0' \leq V \in \C_\mu(L)$ with $\delta(V) = \delta(X)$ and $\vert V \setminus A_0' \vert $ sufficiently large. For example, by our assumption on $X$,  $\C_\mu(L)$ contains arbitrarily large finite algebraic extensions of $X$, so we can take  $V \in \C_\mu(L)$ to be an amalgamation of $A_0'$ and one of these over $X$. Let $A''$ be the free amalgam of $A'$ and $V$ over $A_0'$ and let $B_0'$ be the copy of $V$ inside this. As $A_0'$ is $d$-closed in $A'$ and $A_0' \leq V$ it follows from \ref{aalemma}  that $A'' \in \C_\mu(L)$, $B_0'$ is $d$-closed in $A''$ and $A' \leq A''$. Note that $\delta(A'') = \delta(A')$.

\medskip

\noindent\textit{Step 2:\/} Construction of $B_0' \cup B_i' \in \C_\mu(L)$.\newline
We do this so that $B_0'$ is $d$-closed in $B_0' \cup B_i'$ and $\delta(B_i'\cup B_0'/B_0') = 1$. (As $\delta(B_0' \cup A_i'/B_0') = 1$, it then follows that  $B_0' \cup A_i' \leq B_0'\cup B_i'$ when $1 \leq i \leq r$.) Let $m$ be sufficiently large. Choose some $R_i$: for example $R_1$ of arity $n \geq 3$.  \newline

\smallskip

\textit{Case 1:\/} Suppose $i \leq r$. Pick $b_{i0} \in A_i'$ and let $s_{i1},\ldots, s_{im}$ be disjoint $(n-2)$-subsets  of $B_0'\setminus A_0'$ (we adjust the choice of $V$ in step 1 to accommodate this). Let $B_i' = A_i' \cup \{b_{i1},\ldots, b_{im}\}$ and include as new $R_1$-relations on $B_0' \cup B_i'$ the $n$-sets $\{b_{i0}, b_{ij}\}\cup s_{ij}$ for $1\leq j\leq m$. We need to show that this has the required properties.

First, note that $B_0'\cup A_i' \leq B_0' \cup A_i' \cup \{b_{ij}\}$, so $B_0' \cup A_i' \cup \{b_{ij}\} \in \C_0(L)$. 

Suppose $Y \leq Z$ is a msa extension in $B_0' \cup A_i' \cup \{b_{ij}\}$ not contained in $B_0'\cup A_i'$. So $s_{ij} \cup \{b_{i0},b_{ij}\} \subseteq Z$. If $b_{ij} \not\in Y$ then $Y \leq Z\setminus\{b_{ij}\} < Z$ is algebraic, so $Z\setminus \{b_{ij}\} = Y$ and $Y = s_{ij}\cup \{b_{i0}\}.$ As the elements of $s_{ij}$ are non-adjacent to $b_{i0}$ in $B_0'\cup A_i'$, it follows that there is only one copy of $Z$ over $Y$ in $B_0' \cup A_i' \cup \{b_{ij}\}$. If $b_{ij} \in Y$, then by Lemma \ref{21} $s_{ij}\cup \{b_{i0}\} \not\subseteq Y$. But then there is at most one copy of $Z$ over $Y$ in $B_0' \cup A_i' \cup \{b_{ij}\}$. Note that  $\delta(Y) = \delta(Z) \geq \delta(Z\cap (A_i' \cup B_0')) \geq 1$. So in both cases we meet the requirements for $B_0' \cup A_i' \cup \{b_{ij}\} \in \C_\mu(L)$.

Now note that $B_0' \cup B_i'$ is the free amalgam over $B'_0\cup A_i'$ of the structures $B_0'\cup A_i' \cup \{b_{ij}\}$ (for $j = 1, \ldots, m$). Each $B_0'\cup A_i' \subseteq B_0'\cup A_i' \cup \{b_{ij}\}$ is an algebraic extension and the only msa extension in this with base in $B_0' \cup A_i'$ and which is not contained in $B_0'\cup A_i'$ is $s_{ij}\cup \{b_{i0}\} \leq s_{ij}\cup \{b_{i0}, b_{ij}\}$. So the amalgamation lemma \ref{aalemma} implies that $B_0' \cup A_i' \leq B_0'\cup B_i' \in \C_\mu(L)$. It is clear that $\delta(B_i'/B_0' \cup A_i') = 0$ so $\delta(B_0' \cup B_i') = \delta(B_0' \cup A_i') = \delta(B_0') + 1$.

Finally, note that as $B_0'$ is $d$-closed in $B_0' \cup A_i'$ and $B_0' \cup A_i' \leq B_0' \cup B_i'$, the $d$-closure of $B_0'$ in $B_0'\cup B_i'$ does not contain $b_{i0}$. It then follows easily that $B_0'$ is $d$-closed in $B_0'\cup B_i'$.

\smallskip

\noindent\textit{Case 2:\/} $i > r$. As in Case 1,  let $s_{i1},\ldots, s_{im}$ be  $(n-2)$-subsets  of $B_0'\setminus A_0'$ with no relations on them. Let $B_i' =  \{b_{i1},\ldots, b_{im}\}$ and include as new $R_1$-relations on $B_0' \cup B_i'$ the $n$-sets $\{b_{ij}, b_{i(j+1)}\}\cup s_{ij}$ for $1\leq j\leq m-1$. In this version of the construction we take the $s_{ij}$ to be $s_i$, independent of $j$.

It is clear that $\delta(B_0'\cup B_i'/B_0') = 1$ and if $\emptyset\neq Y \subseteq B_i'$ then $\delta(B_0' \cup Y/B_0') \geq 1$. So $B_0' \leq B_0'\cup B_i'$ and  therefore $B_0'\cup B_i' \in \C_0(L)$, and $B_0'$ is $d$-closed in $B_0' \cup B_i'$. It remains to show that $B_0'\cup B_i' \in \C_\mu(L)$, so suppose $Y \leq Z_1$ is a msa extension in $B_0'\cup B_i'$. If $\delta(Y) = 0$ then $Z_1\subseteq B_0'$ and the same is true of any copy of $Z_1$ over $Y$. Similarly, if $Y \subseteq B_0'$ then all copies of $Z_1$ over $Y$ are contained in $B_0'$ as this is $d$-closed in $B_0'\cup B_i'$. So we can assume that $\delta(Y) \geq 1$ and $b_{ij} \in Y$ for some $j$. By Lemma \ref{21}, we can assume that one of the relations $s_i \cup \{b_{ij},b_{i(j\pm1)}\}$ is a subset of $Z_1$. 

If $Z_1 \setminus Y$ is a singleton then there is at most one other copy $Z_2$ of $Z_1$ over $Y$, and in this case $Y = Z_1 \cap Z_2 = s_i\cup \{b_{ij}\}$. Note that $\delta(Y) \geq 2$ here, so $\mu(Y, Z_1) \geq 2$, by hypothesis.

Now suppose that $Z_1 \setminus Y$ has at least two elements. It will suffice to prove that there is no other copy $Z_2$ of $Z_1$ over $Y$ in $B_0'\cup B_i'$, so suppose there is such a $Z_2$. Take $j$ maximal such that $b_{ij} \in Z_1\cup Z_2$. By Lemma \ref{21}, $b_{ij}$ is in at least two relations in $Z_1\cup Z_2$; but $b_{ij}$ is only in two relations in $B_0'\cup B_i'$ and one of these also involves $b_{i(j+1)}$, so this is in $Z_1\cup Z_2$. This contradicts the choice of $j$.

\medskip

\noindent\textit{Step 3:\/} Other relations on $B'$. \newline
The relations on $B'$ not contained in $A'$ or some $B_0' \cup B_i'$ are $\rho_1', \ldots \rho_t'$. We can choose these to be subsets of $B'\setminus A''$ with   $\rho_i' \cap \rho_j' = \emptyset$ if $i\neq j$, and $\rho_i' \cap B_j' \neq \emptyset$ iff $\rho_i \cap B_j \neq \emptyset$ (for $j \geq 1$). Note that this is possible if $m$ is sufficiently large. We make $\rho_i'$ of the same type as $\rho_i$ (that is, in the same $R_j$). 

This completes the construction of $B'$. We now make a series of claims about it.

\medskip

\noindent\textit{Claim 1:\/} Let $U \subseteq \{1,\ldots, s\}$, and $Y = \bigcup_{i\in U} (B_i\cup B_0)$, $Y' =  \bigcup_{i\in U} (B_i' \cup B_0')$. Then  $Y \cap A$ is $d$-closed in $A$ iff $Y' \cap A'$ is $d$-closed in $A'$, and in this case we have $\delta(Y/B_0) = \delta(Y'/B_0').$

\smallskip

Let $U_0 = \{i\in U : i \leq r\}$ and $U_1 = \{i \in U : i > r\}$. Then $Y\cap A = \bigcup_{i \in U_0} (A_0\cup A_i)$ and $Y' \cap A' = \bigcup_{i \in U_0} (A_0'\cup A_i')$. Because $f$ is an isomorphism of geometries, one of these is $d$-closed (in $A$ or $A'$) iff the other is (remembering that a subset of a geometry is $d$-closed iff any set properly containing it has bigger dimension). So suppose this is the case. We compute that:
\[\delta(Y'/B_0') = \delta(A''\cap Y'/B_0') + \delta(Y'/A''\cap Y') = \delta(A''\cap Y'/B_0') + \vert U_1\vert - \vert J\vert\]
where $J = \{ j : \rho_j' \subseteq Y'\}$. This follows from the fact that $Y'$ consists of $\vert U_0\vert$ sets of $\delta$-value $0$ over $A''\cap Y'$ and $\vert U_1\vert$ sets of $\delta$-value $1$ over $A''\cap Y'$, and an extra $\vert J\vert$ relations $\rho'_j$ between them. Moreover
\[\delta(A'' \cap Y'/B_0') = \delta ((A'\cap Y')\cup B_0'/B_0') = \delta(A' \cap Y'/A_0'),\]
using, for example, the construction of $A''$ as a free amalgam in step 1. Thus we have
\[\delta(Y'/B_0') = \delta(A'\cap Y'/A_0') + \vert U_1\vert - \vert J\vert.\]
Now, by construction (step 3) we have $\rho_j \subseteq Y$ iff $\rho'_j \subseteq Y'$. So an identical calculation shows that 
\[\delta(Y/B_0) = \delta(A\cap Y/A_0) + \vert U_1\vert - \vert J\vert.\]
(This uses the fact that $A, B_0$ are freely amalgamated over $A_0$, which follows from the definitions of $A_0$, $B_0$ and the assumption that $A\leq B$.)
By the isomorphism $f$, we have $d(A'\cap Y'/X) = d(A\cap Y/W)$. If  $A'\cap Y'$, $A\cap Y$ are $d$-closed (in $A'$, $A$ respectively) then $d(A'\cap Y') = \delta(A'\cap Y')$ and $d(A\cap Y) = \delta(A \cap Y)$.  So in this case $\delta(A\cap Y/ A_0) = \delta(A' \cap Y' / A_0')$, as $d(W) = \delta(A_0)$ and $d(X) =\delta(A_0')$. Thus we have  $\delta(Y/B_0) = \delta(Y'/B_0')$, as required. \hfill$\Box_{Claim}$

\medskip

\noindent\textit{Claim 2:\/} $B' \in \C_0(L)$, $B_0'$ and $B_0'\cup B_i'$ are $d$-closed in $B'$ and $A'' \leq B'$. The map $f' : G_W(B) \to G_X(B')$ given by $f'(B_i) = B_i'$ is an isomorphism of geometries which extends $f$. 

\smallskip

First, note that (by construction step 2) if $B_0'\subseteq C \subseteq B'$ and $Y' = \bigcup\{ B_0'\cup B_i' : C \cap B_i' \neq \emptyset\}$, then $\delta(C) \geq \delta(Y')$. If $A'' \subseteq C$ then by Claim 1, $\delta(Y'/B_0') = \delta(Y/B_0)$, where $Y = \bigcup\{ B_0\cup B_i : C \cap B_i' \neq \emptyset\}$. This contains $A$, so as $A \leq B$ we have $\delta(Y) \geq \delta(A)$. By the isomorphism, $\delta(A/A_0) = \delta(A'/A_0') = \delta(A''/B_0')$ (by step 1). Thus $\delta(C/B_0') \geq \delta(Y'/B_0') = \delta(Y/B_0) \geq \delta(A/A_0) = \delta(A''/B_0')$ (using that $\delta(A_0) = \delta(B_0)$). This shows $A''\leq B'$ and therefore (as $\emptyset \leq A''$) we also have $\emptyset \leq B'$.

Now suppose $Y' \geq X$ is $d$-closed in $B'$. Then $B_0' \subseteq Y'$ and as  above, $Y'$ is of the form $\bigcup_{i \in U} (B_0'\cup B_i')$ for some $U$. Moreover $Y'\cap A'$ is $d$-closed in $A'$ and so we can apply Claim 1. It follows from this  that $B_0'$ is $d$-closed in $B'$ and the $d$-closed sets of dimension 1 over $B_0'$ are the $B_0'\cup B_i'$, by using the fact that the corresponding statements hold in $B$. 

It remains to show that $f'$ is an isomorphism of geometries. Let $Y$, $Y'$ be as in Claim 1. We need to show that $Y$ is $d$-closed in $B$ iff $Y'$ is $d$-closed in $B'$. So suppose $Y$ is $d$-closed in $B$. Then $Y\cap A$ is $d$-closed in $A$ and so we can apply Claim 1 to get that $\delta(Y/B_0) = \delta(Y'/B_0')$. Suppose for a contradiction that $Y'$ is not $d$-closed in $B'$. Let $Z'$ be its $d$-closure in $B'$. So $Z' = \bigcup_{i \in Q} (B_0'\cup B_i')$ for some $Q$ with $U \subset Q \subseteq \{1,\ldots, s\}$ and  $\delta(Z') \leq \delta(Y')$. Let $Z = \bigcup_{i \in Q} (B_0\cup B_i)$. So $Y \subset Z$. Because $Z'$ is $d$-closed in $B'$ and therefore $Z'\cap A'$ is $d$-closed in $A'$, we can apply Claim 1 to get that $\delta(Z/B_0) = \delta(Z'/B_0')$. So we have
\[\delta(Z/B_0) = \delta(Z'/B_0') \leq \delta(Y'/B_0') = \delta(Y/B_0)\]
and this contradicts the fact that $Y$ is $d$-closed and $Y \subset Z$. Thus $Y'$ is $d$-closed in $B'$. The argument for the converse implication is the same.\hfill$\Box_{Claim}$

\medskip

\noindent\textit{Claim 3:\/} $B' \in \C_\mu(L)$.

\smallskip

Suppose that $Y \leq Z$ is a minimally simply algebraic extension in $B'$.  First suppose $\delta(Y) = \delta(Z) \leq 1$. Then $d(Y) \leq 1$, so $Y \subseteq B_0'\cup B_i'$ for some $i$ and as $B_0'\cup B_i'$ is $d$-closed in $B'$, any copies of $Z$ over $Y$ in $B'$ are contained in $B_0'\cup B_i'$. So there are at most $\mu(Y,Z)$ of these as $B_0'\cup B_i' \in \C_\mu(L)$. 

Now suppose that $\delta(Y) \geq 2$ and suppose for a contradiction  that $Z_i$ (for $i = 1,\ldots, \mu(Y,Z)+1$) are  disjoint copies in $B'$ of $Z$ over $Y \subseteq B'$ (meaning that the sets $Z_i \setminus Y$ are disjoint, of course). 

If $y \in Y$ then $y$ is in some relation in $R_k^Z\setminus R_k^Y$ (for some $k \in I$) by Lemma \ref{21}. Thus $y$ is  in at least three relations in $R_k^{B'}$  (one in each $R_k^{Z_i}\setminus R_k^Y$). By inspection of the construction one therefore sees that $y \in A''$ or $y = b_{ij}$ for some $i > r$. In the latter case, two of the (at most) three relations in $B'$ which involve $b_{ij}$ are $s_{i}\cup \{b_{i(j-1)}, b_{ij}\}$ and $s_{i}\cup \{b_{i(j+1)}, b_{ij}\}$. So we can assume that the first is a subset of $Z_1$ (and not a subset of $Y$) and the second is a subset of $Z_2$. But this implies that $s_{i} \subseteq Y$. However, there is no other relation which contains $\{b_{ij}\} \cup s_{i}$: contradicting the fact that $Z_3$ is a copy of $Z_1$ over $Y$. 

Thus $Y \subseteq A''$. As $A'' \in \C_\mu(L)$ not all of the $Z_i$ are subsets of $A''$, so we can assume that $Z_1 \not\subseteq A''$. As $A'' \leq B'$ we have $Y \subseteq A'' \cap Z_1 \leq Z_1$ so (by the simplicity of the extension) $Z_1 \cap A'' = Y$. 

Note that $Z_1$ is in the $d$-closure of $Y$ so we cannot have $Y \subseteq B_0'$. Let $y \in Y \setminus B_0'$. This is adjacent in $Z_1$ to some $z \in Z_1\setminus Y$. So $y \in A'' \setminus B_0'$ is adjacent in $B'$ to $z \in B'\setminus A''$. Inspection of the construction shows that $y = b_{i0}$ (for some $i \leq r$) and $z = b_{ij}$. Then  the adjacency of $y$ and $z$ in $Z_1$ forces $s_{ij} \subseteq Z_1$, and so $s_{ij} \subseteq Y$ (as $A''\cap Z_1 = Y$). But then $Y \leq Y \cup \{b_{ij}\}$ is a simply algebraic extension in $Z_1$. As $Y \leq Z_1$ is a minimally simply algebraic extension, this implies $Y = \{b_{i0}\}\cup  s_{ij}$  and $Z_1  = \{b_{i0}, b_{ij}\}\cup s_{ij}$. However, there is no other relation in $B'$ which contains this $Y$ (by construction), so we have a contradiction. \hfill$\Box_{Claim}$

\medskip

Claims 2 and 3 finish the proof of the isomorphism extension property  $\C_0(L) (W)\rightsquigarrow \C_\mu(L)(X)$.

\medskip

For the other direction, we can use the same construction (it is a special case of the the above as $\C_\mu(L) \subseteq \C_0(L)$). Of course, in this case we do not need Claim 3.
\end{proof}

\section{Further isomorphisms}
\subsection{Localization of non-isomorphic geometries}
In this subsection the language $L$ has just a  single $3$-ary relation $R$. We often suppress $L$ in the notation.

In 5.2 of \cite{EH} Hrushovski varies his strongly minimal set construction to produce examples where the model-theoretic structure of the strongly minimal set can be read off from the geometry: lines of the geometry have three points, and colinear points correspond to instances of the ternary relation. He thereby produces continuum-many non-isomorphic geometries of (countable, saturated) strongly minimal structures, but asks whether these examples are \textit{locally} isomorphic. We show that this is the case: in fact, localizing any of them over a 3-dimensional set gives a geometry isomorphic to $G(\M_0(L))$, the geometry of the generic structure for $(\C_0(L), \leq)$.

 In 5.2 of \cite{EH}, Hrushovski considers
\[\K_0 = \{A \in \C_0(L) : B \leq A \mbox{ for all $B \subseteq A$ with $\vert B \vert \leq 3$}\}.\]

The class $(\K_0, \leq)$ is an amalgamation class: one shows that if $A \leq B_1, B_2 \in \K_0$ and the free amalgam of $B_1$ and $B_2$ over $A$ is not  in $\K_0$, then there exist $a, a' \in A$ and $b_i \in B_i\setminus A$ with $R(a,a', b_i)$ holding in $B_i$ (for $i=1,2$).

More generally, given a function $\mu$ as before, we can consider $\K_\mu = \K_0 \cap \C_\mu(L)$ and for appropriate $\mu$, the class $(\K_\mu, \leq)$ will satisfy Assumption \ref{aalemma}. In fact, we only need to define $\mu(Y,Z)$ for $\delta(Y) \geq 3$. For suppose $Y \leq Z$ is a minimally simply algebraic extension in $\K_0$ and $\delta(Y) = \delta(Z) \leq 2$. Then $Z$ has at most 3 elements: otherwise there is a subset $W\subseteq Z$ of size 3 with $W \not\in R^Z$, and then $W$ is not self-sufficient in $Z$, contradicting the definition of $\K_0$. It follows that the value of $\mu(Y,Z)$ is irrelevant for such $Y \leq Z$: the multiplicity is already controlled by the definition of $\K_0$. 

So we shall assume: 

\begin{ass}\label{HAmalgLemma} With the above notation:
\begin{enumerate}
\item[(i)] If $A \leq B_1, B_2 \in \K_\mu$ and the free amalgam of $B_1$ and $B_2$ over $A$ is not  in $\K_\mu$, then either there exist $a, a' \in A$ and $b_i \in B_i\setminus A$ with $R(a,a', b_i)$ holding in $B_i$ (for $i=1,2$), or there exists $Y \subseteq A$ with $\delta(Y) \geq 3$ and msa extensions $Y \leq Z_i \in B_i$ (for $i= 1,2$) which are isomorphic over $Y$ and $Z_i\setminus Y \subseteq B_i\setminus A$.
\item[(ii)]  The class $(\K_\mu, \leq)$ is an amalgamation class.
\end{enumerate}
\end{ass}

Again, (ii) follows from (i) here and the condition $\mu(Y,Z) \geq \delta(Y)$ (for $\delta(Y) \geq 3$) guarantees that (i) holds.

Denote the generic structure of $(\K_\mu, \leq)$ by $\cN_\mu$.  The $d$-closure of two points in $\cN_\mu$ has size $3$ (as above), so certainly $G(\cN_\mu)$ and $G(\M_0(L))$ are non-isomorphic. In fact,  we can recover the relation $R$ from the geometry $G(\cN_\mu)$ as the $3$-sets with dimension 2. Thus different $\mu$ give different geometries. It can be shown that there are infinitely many msa extensions $Y \leq Z \in \K_0$ with $\delta(Y) = 3$, so if $\mu(Y, Z) \geq 1$ for infinitely many of these, then $\cl_{\cN_\mu}(X)$ is infinite whenever $X \leq \cN_\mu$ consists of 3 independent points.

We show:

\begin{thm}\label{surprise}
Suppose $\mu(Y,Z) \geq 3$ for all msa $Y \leq Z$ in $\K_0$ with $\delta(Y) \geq 3$ and Assumption \ref{HAmalgLemma} holds.  Let $X \leq \cN_\mu$ and $d(X) = 3$. Then the localization $G_X(\cN_\mu)$ is isomorphic to $G(\M_0(L))$. 
\end{thm}

\begin{proof}
We may assume that $X$ consists of 3 points (and no relations). By the remarks preceding the theorem, $X$ has arbitrarily large finite algebraic extensions in $\K_\mu$. We show that the isomorphism extension property $\C_0(L) \rightsquigarrow \K_\mu(X)$ holds.

Suppose we are given $A \leq B \in \C_0(L)$ and $X \leq A' \in \K_\mu$ and  an isomorphism $f : G(A) \to G_X(A')$. We want to find $B' \in \K_\mu$ with $A' \leq B'$ and an isomorphism $f' : G(B) \to G_X(B')$ extending $f$. This is very similar to the the construction of $B'$ in the proof of Theorem \ref{main} and we will only indicate what needs to be modified and provide extra argument as required.

Let $A_0 = \cl_A(\emptyset)$ and let $A_1,\ldots, A_r$ be the $d$-dependence classes on $A\setminus A_0$: the latter are the points of $G(A)$. Similarly let $B_0 = \cl_B(\emptyset)$ and $B_1, \ldots, B_s$ the $d$-dependence classes on $B\setminus B_0$, with $A_i \subseteq B_i$ for $i = 1,\ldots, r$. List the relations on $B$ which are not contained in $A$ or some $B_0\cup B_j$ as $\rho_1,\ldots, \rho_t$. So these are 3-sets and note that each of them intersects three different $B_i$. Let $A_0' = \cl_{A'}(X)$ and $A_1',\ldots, A_r'$ be the classes of $d$-dependence over $X$ on $A'\setminus A_0'$, labelled so that $f(A_i) = A_i'$. We construct $B_0', B_1', \ldots, B_s'$ with $A_i' \subseteq B_i'$ for $i = 0, \ldots, r$, and $B' = \bigcup_{i = 0}^s B_i'$ in the following way.

\medskip

\noindent\textit{Step 1:\/} Construction of $A'' = A' \cup B_0' \in \K_\mu$. \newline
This is as before, but we need to take $V \in \K_\mu$: we can do this because algebraic extensions of $X$ can be arbitrarily large.

\medskip

\noindent\textit{Step 2:\/} Construction of $B_0' \cup B_i' \in \K_\mu$.\newline
The construction is as in Theorem \ref{main} for $i \leq r$. In the case $i > r$ we vary the construction by taking the $s_{ij}$ to be distinct. The proofs that $B_0'$ is $d$-closed in $B_0'\cup B_i'$ are as before; as are the arguments which show that if $Y \leq Z$ is a msa extension in $\K_0$ with $\delta(Y) \geq 3$ then there are at most $\mu(Y,Z)$ copies of $Z$ over $Y$ in $B_0'\cup B_i'$. So it remains to show that $B_0' \cup B_i' \in \K_0$. 

If $i \leq r$, then using the amalgmation lemma \ref{HAmalgLemma} as in Step 2, Case 1 of Theorem \ref{main}, it will suffice to show that $B_0'\cup A_i' \cup \{b_{ij}\} \in \K_0$. This is the free amalgam of $\{s_{ij}, b_{i0}, b_{ij}\}$ and $B_0'\cup A_i'$ over $\{s_{ij},b_{i0}\}$. So we can apply \ref{HAmalgLemma} (because $\{s_{ij},b_{i0}\}$ is in no relation in $B_0'\cup A_i'$).

Now suppose $i > r$. We analyse the possibilities for $\delta(Y)$ when $Y \subseteq B_0'\cup B_i'$, $Y \not\subseteq B_0'$ and $\vert Y \vert > 1$. As $Y \cap B_0'$ is $d$-closed in  $Y$ we have $\delta(Y) > \delta(Y \cap B_0')$. If $Y \cap B_0' = \emptyset$ then by the construction, $\delta(Y) = \vert Y \vert$. If $Y \cap B_0'$ is a singleton then $Y$ has at most one relation (because the $s_{ij}$ are distinct), so $\delta(Y) > 1$ and $\delta(Y) \geq \vert Y \vert -1$. In the remaining case, $\delta(Y \cap B_0') \geq 2$ (as $B_0' \in \K_0$), so $\delta(Y) \geq 3$. Thus, $Y$ consists of 2 points, or is 3 points in a relation, or has $\delta(Y) \geq 3$. It follows that $B_0'\cup B_i' \in \K_0$.

\medskip

\noindent\textit{Step 3:\/} Other relations on $B'$. \newline
As before. 

\medskip

\noindent\textit{Claim 1:\/} Let $U \subseteq \{1,\ldots, s\}$, and $Y = \bigcup_{i\in U} (B_i\cup B_0)$, $Y' =  \bigcup_{i\in U} (B_i' \cup B_0')$. Then  $Y \cap A$ is $d$-closed in $A$ iff $Y' \cap A'$ is $d$-closed in $A'$, and in this case we have $\delta(Y) = \delta(Y'/B_0').$

As before.

\medskip
\noindent\textit{Claim 2:\/}
$B' \in \C_0(L)$, $B_0'$ and $B_0'\cup B_i'$ are $d$-closed in $B'$ and $A'' \leq B'$. The map $f' : G(B) \to G_X(B')$ given by $f'(B_i) = B_i'$ is an isomorphism of geometries which extends $f$. 

As before, using Claim 1.

\medskip

\noindent\textit{Claim 3:\/} If $i\neq j$ then $B_0'\cup B_i' \cup B_j' \in \K_\mu$ and $B_0' \cup B_i'\cup B_j' \leq B'$.

\smallskip

By construction $B_0'\cup B_i' \cup B_j'$  is the free amalgam of $B_0'\cup B_i'$ and $B_0' \cup B_j'$ over $B_0'$, and $B_0'$ is $d$-closed in each. So the first statement follows from the assumed amalgamation lemma \ref{HAmalgLemma}. We have $\delta(B_0' \cup B_i'\cup B_j'/ B_0') = 2$. Moreover, as $B_0'\cup B_i'$ is $d$-closed in $B'$ (Claim 2), if $B_0' \cup B_i'\cup B_j'\subseteq Z$ then $\delta(Z/B_0') \geq 2$. This gives the second statement.\hfill$\Box_{Claim}$

\medskip

\noindent\textit{Claim 4:\/} $B' \in \K_0$.

\smallskip

We need to show that if $D \subseteq B'$ has size at most 3 then $D \leq B'$. If $\vert D \vert \leq 2$ then $D \subseteq B_0' \cup B_i' \cup B_j'$ for some $i, j$ and it follows from Claim 3 that $D \leq B_0'\cup B_i'\cup B_j' \leq B'$. So suppose $D$ has size 3 and $D \subseteq C$ with $\delta(C) < \delta(D)$. We must have $\delta(C) = 2$ (as any two points of $D$ are self-sufficient in $C$ and have $\delta$-value $2$). As $A'' \in \K_0$ there is an $i$ such that $C \cap (B_i'\setminus A'') \neq \emptyset$. Note that  $C$ is not contained in $B_0'\cup B_i'$ (because this is in $\K_\mu$), so as $B_0'$ is $d$-closed in $B_0' \cup B_i'$ and the latter is $d$-closed in $B'$, we have 
\[0 \leq \delta(C\cap B_0') < \delta(C \cap (B_0'\cup B_i')) < \delta(C) = 2.\]
Thus $\delta(C\cap B_0') = 0$, so $C \cap B_0' = \emptyset$.

It then follows from  Step 2 of the construction that there is no adjacency in $C$ between points of $C\cap A''$ and points of $C \setminus A''$. Let $q = \vert\{j : \rho_j' \subseteq C\}\vert$. Then (using $A'' \leq B'$; so $C \cap A'' \leq C$)
\[2 \geq \delta(C/C\cap A'') = \vert C \setminus A''\vert - q \geq 2q\]
as the $\rho_j'$ are disjoint. If $q = 0$ then $C$ is $C\cap A''$ together with some isolated points, and this is in $\K_\mu$ (so not possible in this situation). If $q = 1$ then $C$ consists of 3 points in a single relation and this has no subset of the form required for $D$. \hfill$\Box_{Claim}$

\medskip

\noindent\textit{Claim 5:\/} $B' \in \K_\mu$\newline
 We already know that $B' \in \K_0$, so we need to show that $B' \in \C_\mu(L)$, at least as far as msa extensions $Y \leq Z$ with $\delta(Y) \geq 3$ are concerned. So suppose $Z_1,\ldots, Z_4$ are disjoint copies of $Z$ over $Y$ in $B'$. Then each element  $y \in Y$ is in at least $4$ relations (using \ref{21}, as before), so by construction, $y \in A''$. Thus $Y \subseteq A''$ and the rest of the proof is just as in Claim 3 of Theorem \ref{main}. \hfill$\Box_{Claim}$

 \medskip

Claims 2 and 5 finish the proof of one direction of the isomorphism extension property. The direction $\K_\mu(X) \rightsquigarrow \C_0(L)$ follows from the property $\C_\mu(L)(X) \rightsquigarrow \C_0(L)$ proved in Theorem \ref{main}.

 \end{proof}

 \begin{rem}\rm Note that $\K_0$ can be considered as $\K_\mu$ where $\mu(Y,Z)$ is formally given the value $\infty$ for all msa $Y \leq Z \in \K_0$. Thus the above argument also shows that the geometry of $\cN_0$, the generic for $(\K_0, \leq)$, is locally isomorphic to $G(\M_0(L))$.
 \end{rem}

\begin{rem} \rm Another variation is given in 5.1 of \cite{EH}. Let $k \geq 2$ and consider the language $L$ with a single $(k+1)$-ary relation symbol $R$. Let 
 \[\C_0'(L) = \{A \in \C_0(L): \delta(B) \geq\min(\vert B\vert , k)\,\, \forall B \subseteq A\}.\]
So if $C \subseteq A \in \C_0'(L)$ and $\vert C \vert \leq k$ then $C \leq A$. Hrushovski observes that $(\C_0'(L), \leq)$ is a free amalgamation class and that the assumed amalgamation lemma \ref{aalemma} holds for $(\C_\mu', \leq)$, for suitable $\mu \geq 2$. The generic structures here are strongly minimal and any $k$ points are independent. So the geometries are again different from that of $\M_0(L)$. However, they are again locally isomorphic. To see this we proceed as in Theorem \ref{surprise}, but take $X$ to be a set of size $k$. The construction and proof are essentially the same as before, except for in Claim 4 where to show that $B' \in \C_0'(L)$ we modify the argument as follows.

Suppose $C \subseteq B'$ has $\delta(C) < k$ and $\vert C \vert \geq k+1$. Then for some $i$ we have:
\[0 \leq \delta(C\cap B_0') < \delta(C\cap (B_0' \cup B_i')) <\delta(C) \leq k-1.\]
So $\delta(C\cap B_0') \leq k-3$ and therefore $\vert C \cap B_0'\vert \leq k-3$. Then by construction there is no adjacency in $C$ between points of $C \cap A''$ and points of $C\setminus A''$. So (with $q$ as before):
\[k-1 \geq \delta(C/C\cap A'') = \vert C \setminus A''\vert - q\geq kq.\]
Thus $q = 0$ and we have a contradiction. 
\end{rem}

 \subsection{Changing the language and predimension}
 Recall that the language $L$ consists of relation symbols $\{R_i: i \in I\}$ with $R_i$ of arity $n_i$ (and only finitely many symbols of each arity). Suppose that $L_0 = \{ R_i : i \in I_0\}$ is a sublanguage with the property that for every $i \in I$ there is $j \in I_0$ such that $n_i \leq n_j$. For example, if $I$ is finite we can take $L_0$ to consist of a relation symbol of maximal arity in $L$.  The following is essentially Theorem 3.1 of \cite{MFDE1}, but working with sets rather than tuples: we omit most of the details of the proof.
 
 \begin{thm}
 The geometries $G(\M_0(L))$ and  $G(\M_0(L_0))$ are isomorphic.
 \end{thm}
 
 \begin{proof}
 We can use the construction in Theorem \ref{main} to show that $\C_0(L) \rightsquigarrow \C_0(L_0)$ holds. In step 3 of the construction, if $\rho_i$ is a $k$-set then we take $\rho_i'$ to be a $k'$-set with $k' \geq k$: the condition on $L_0$ allows us to do this. Claims 1 and 2 of Theorem \ref{main} then go through exactly as previously.  The direction $\C_0(L_0) \rightsquigarrow \C_0(L)$ follows as in Theorem \ref{main}.
 \end{proof}
 
 \begin{rem}\rm
 Theorem 3.1 of \cite{MFDE1} works with a predimension of the form:
 \[\delta_\alpha(A) = \vert A \vert - \sum_{i \in I} \alpha_i\vert R_i^A\vert,\]
 where the $\alpha_i$ are natural numbers. We can adapt the argument here to deal with such predimensions. For example, suppose $I$ is finite and $R_1$ is of maximal arity and $\alpha_1 = 1$. Let $L_0$ consist of $R_1$. Then, as in Theorem 3.1 of \cite{MFDE1}, $G(\M_0(L))$ is isomorphic to $G(\M_0(L_0))$. To show that $\C_0^\alpha(L) \rightsquigarrow \C_0(L_0)$ (where $\C_0^\alpha(L)$ is defined using the predimension $\delta_\alpha$) we perform the same construction except that in step 3,  if $\rho_j$ is of type $R_i$ then we add $\alpha_i$ corresponding $\rho_j'$ (but still disjoint etc).
 
\end{rem}
 
 \subsection{Sets versus tuples} We have chosen to work with structures $A$ where the relations $R_i^A$ are sets of $n_i$-sets. As was done in \cite{MFDE1} we could also have worked more generally with structures $A$ where the $R_i^A$ are sets of $n_i$-tuples and the predimension is still given by $\vert A \vert - \sum _i \vert R_i^A\vert$. Let $\hC_0(L)$ denote the class of these finite structures with $\emptyset \leq A$. 
 
 \begin{thm}
The geometries of the generic structures of the amalgamation classes $(\C_0, \leq)$ and $(\hC_0(L), \leq)$ are isomorphic.
 \end{thm}
 
 \begin{proof}
 This is the usual sort of proof using the construction. For example, to show $\hC_0(L) \rightsquigarrow \C_0(L)$ we replace an  $n_i$-tuple $\rho_j$ (in $R_i^B\setminus R_i^{A}$) by an $n_i$-set, using the new $d$-dependent points to eliminate repetitions of points in the tuple or different enumerations of the same set. 
 \end{proof}

\end{document}